\documentclass[a4paper,11pt]{amsart}

\usepackage{amsmath,amsfonts,amssymb,amsthm,mathtools} % AMS
\mathtoolsset{showonlyrefs=true} % Показывать номера только у тех формул, на которые есть \eqref{} в тексте.
\usepackage{enumitem}
\usepackage{ytableau}
\usepackage{tikz-cd}
\usetikzlibrary{decorations.pathmorphing} 
\usepackage{soul}

\usepackage{hyperref} %надо загружать здесь

\theoremstyle{definition}
\newtheorem{theorem}{Theorem}
\newtheorem{conjecture}[theorem]{Conjecture}

\newtheorem{corollary}[theorem]{Corollary}
\newtheorem{lemma}[theorem]{Lemma}

\newtheorem{claim}[theorem]{Claim}
\newtheorem{definition}[theorem]{Definition}

\newtheorem{remark}[theorem]{Remark}

\mathcode`<="013C
\let\relprec\prec
\renewcommand{\prec}{{\relprec}}
\let\relll\ll
\renewcommand{\ll}{{\relll}}

\title[Towards a categorical analogue of Gelfand-Kazhdan Theorem]{Towards a categorical analogue of Gelfand-Kazhdan Theorem}

\author[Alexander Popkovich]{Alexander Popkovich}

\begin{document}

\begin{abstract}
 A celebrated theorem by Gelfand-Kazhdan states that the restriction of any cuspidal irreducible representations of $GL_n(\mathcal{K})$ over local field to the mirabolic subgroup $P$ is isomorphic to the standard irreducible representation of $P$. We formulate a conjecture that an analogous statement should hold for categorical representations. In this note we prove this for a particular example of an irreducible cuspidal categorical representation of $PGL_2(\mathcal{K})$.
\end{abstract}

\maketitle

\section*{Introduction}

Let $\mathcal{K}$ be a non-Archimedean local field and $G = GL_n(\mathcal{K})$.

Let $P \subset G$ be the mirabolic subgroup, that is the subgroup formed by matrices of the form $M=(m_{i,j})_{i,j=1...n}$ with $m_{in}=0$ for $i<n$ and $m_{nn}=1$.

In \cite{gk} the authors suggested an approach to studying smooth complex representations of $G$ based on restricting them to $P$ (see also \cite{bz} and \cite{bz2} for further discussions of this approach). In particular, one of the key theorems is as follows.

Denote (following \cite{gk}) by $C^\infty(G)$ the space of functions $f:G(\mathcal{K}) \to \mathbb{C}$ such that the following two conditions hold:
\begin{enumerate}
    \item for every $f\in C^\infty(G)$ there exist a congruence subgroup $G_n \subset G$ such that $f$ is invariant under the action of $G_n$ by the right-shift operators
    \item there exists a compact set $K$ such that $supp(f) \subset K\cdot Z(G)$ where $Z(G)$ is the center of $G$.
\end{enumerate}

Let $\psi: \mathcal{K} \to \mathbb{C}^\times$ be any nontrivial character of the additive group of $\mathcal{K}$ and let $\theta:U_n(\mathcal{K}) \to \mathbb{C}^\times$ be the character of the subgroup of upper-triangular matricesl $U_n(\mathcal{K}) \subset GL_n(\mathcal{K})$ that sends a matrix $U=(u_{ij})_{ij}$ to 
\begin{equation}
    \theta(U) = \psi(u_{1\;2}+...+u_{n\;n-1})
\end{equation}
Denote by $C^{\infty}_\theta(P)$ the subspace of functions from $C^{\infty}(P)$ that satisfy the condition $f(u^{-1} p) = \theta(u) f(p)$ for $u\in U_n(\mathcal{K})$ and $p \in P$.

Let $(\pi,V)$ be a nondegenerate representatioin of $G$ and let $\ell$ be a functional on $V$ such that $\ell(\pi(u)v)= \theta(z)l(v)$ for any $u\in U_n(\mathcal{K})$ and $v \in V$. Such $\ell$ provides a mapping $\phi: V \to C_\theta^{\infty}(P)$ sending a vector $v \in V$ to the corresponding matrix coefficient, i.e. the function $f_{v,\ell}$ such that  $f_{v,\ell}(p)= \ell(\pi(p)v)$. Such $\phi$ is an intertwining operator between representations of $P$.

\begin{theorem}
\label{gk_theorem}
    (Theorem 5 of \cite{gk})  If $V$ is an irreducible cuspidal representation of $G$ then $\phi$ is a $P$-equivariant isomorphism between the space V and the space $C^\infty_\theta(P)$. 
\end{theorem}

Denote by $G' \simeq GL_{n-1}(\mathcal{K})$ the subgroup of $P$ formed by matrices of the form $M=(m_{i,j})_{i,j=1...n}$ with $m_{in}=\delta_{in}$ and $m_{nj}=\delta_{nj}$. Note that the restriction map $res_{G'}: C^{\infty}_\theta(P) \to C_\theta^{\infty}(G')$ is an isomorphism due to the decomposition $P=U \cdot G' \cdot U$ and the condition $f(u^{-1} p) = \theta(u) f(p)$. Therefore we have the following:

\begin{corollary}
\label{gk_corollary}
Let $V$ is an irreducible cuspidal representation of $GL_n(\mathcal{K})$, then restrictiom to to representation of $GL_{n-1}(\mathcal{K})$ gives
\begin{equation}
    V|_{GL_{n-1}(\mathcal{K})}= ind_{U_n}^{P_n} \theta = ind_{U_{n-1}}^{P_{n-1}} \theta|_{GL_{n-1}(\mathcal{K})}
\end{equation}
(here and below by $ind$ we mean the functor of induction with compact support). In particular, for $n=2$ we get
\begin{equation}
    V|_{GL_1} \simeq C^\infty(GL_1(\mathcal{K}))
\end{equation}
\end{corollary}

One of the applications of this theorem suggested in \cite{gk} was a construction of a gamma function of a pair $V$,$W$ where $V$ is a representation of $GL_n$ and $W$ is a representation of $GL_{n-1}$. 

The reason for our interest in this construction is motivated by the local geometric Langlands conjecture, which studies categories with (strong, de Rham) action of a reductive group and relates them with local systems for dual Langlands group over the formal punctured disk. This conjecture is widely open. In order to get a better understanding of objects in question we wish to obtain an analogue of the Theorem \ref{gk_theorem} in the categorical context. 

In this new context, it requires some effort to formulate the conditions of the theorem - for example, the notion of irreducibility of a categorical representation is not immediately clear. However, there is definition of a cuspidal irreducible representation which we use below and believe to be sensible. Namely, we say that a strong categorical representation $\mathcal{C}$ of $G$ (which is now defined over complex numbers) is irreducible cuspidal if it is cuspidal (that is, the invariants of the action of $U(\mathcal{K})\subset G$ are trivial) and the Whittaker invariants functor provides an equivalence between $\mathcal{C}$ and the category of vector spaces. This definition mimics the fact that in the case of the usual representations of reductive groups over local fields a cuspidal representation is irreducible if and only if the space of Whittaker invariants is one-dimensioinal. In fact, one can show for $G=GL_2$ that the Whittaker functor is nontrivial on any cuspidal categorical representation of $G$ (and according to a private communication with Sam Raskin this should also be true for any reductive $G$), so if the Whittaker invariants of a categorical representation are $Vect$ then it makes sense to believe it to be irreducible.

For any irreducible cuspidal categorical representation $\mathcal{C}$ one can construct a functor to $D(P)_{/(U,\theta)}$ (the category of $D$-modules on the group $P$ in the sense of \cite{r1} which are $U$-equivariant against a character $\theta$) following the same template as in the case of the mapping $\phi$ defined above. Namely, let $Whit: \mathcal{C} \to Vect$ be the Whittaker functor and $c \in Ob(\mathcal{C})$ be an object of $\mathcal{C}$. Then we have a "functorial matrix coefficient" $f_{c}: D(P) \to Vect$ given by the rule $\mathcal{F} \mapsto Whit(\mathcal{F}.c)$ for $\mathcal{F} \in Ob(D(P))$. But the category $D(P)$ is self-dual (this follows from Proposition 5.5.1 of \cite{r1}) so such functor $f_c$ must be equivalent to $Hom(\cdot,\mathcal{F}_c)$ for some $D_c \in Ob(D(P))$. The correspondence $c \mapsto \mathcal{F}_c$ defines a functor $\Phi: \mathcal{C} \to D(P)$.

\begin{conjecture}
\label{the_conjecture}
    Let $\mathcal{C}$ be an irreducible cuspidal categorical representation of $G$, that is a category equipped with a strong actioin of $G$ such that the Whittaker invariants functor establishes an equivalence $\mathcal{C} \simeq Vect$. Then the functor  $\Phi: \mathcal{C} \to D(P)_{/(U,\theta)}$ defined above is an equivalence of categories.
\end{conjecture}

We don’t know how to prove the hypothesis generally. The goal of this note is to treat an example of an irreducible cuspidal categorical representation in which we can show that this functor is an equivalence. For simplicity purposes we consider a representation of the group $PGL_2(\mathcal{K})$, which is equivalent to working with a representation of $GL_2(\mathcal{K})$ with a trivial simple character. Namely, we consider categorical analogue of an irreducible cuspidal representation induced from a generic character of $I^0$, where $I^0$ is the unipotent radical of the Iwahori subgroup. We will show that the functor discussed above provides an equivalence between our categorical representation and the category of $D$-modules on $\mathcal{K}^\times$.

Let us note one of the reasons why potentially this is less trivial than in the functional case: in this context it is not clear why should $D(\mathcal{K}^ \times)$ and not $D(\mathcal{K} \setminus 0)$ appear - while $\mathcal{K}^\times$ and $\mathcal{K} \setminus 0$ are classically the same as sets, they represent different functors as schemes.

\begin{remark}
    We recently learned from a private communication with David Yang that there is a work in progress by Eunsu Hur concerning the classification of all irreducible cuspidal categorical representations of $GL_2(\mathcal{K})$. It turns out that they all are equivalent to those of the form $D(GL_2(\mathcal{K}))_{/(K,\chi)}$ where $K \subset GL_2(\mathcal{K})$ is a 
    %induced from a character
    compact open subgroup and $\chi:K\to \mathbb{C}^\times$ is it's character. We therefore hope that our methods can be used to prove the Conjecture \ref{the_conjecture} for all irreducible cuspidal representations of $GL_2(\mathcal{K})$.
\end{remark}
\section*{Acknowledgments}
This work would not have been possible without the guidance of Alexander Braverman. I owe a great debt to him for posing the problem and for countless helpful discussions.

I am also grateful to Ilya Dumanski, who took on the heroic task of reading this text's draft and corrected many errors.  Any flaws that remain are entirely my responsibility.

\section{Functional case}

In this section we work out the classical case: we consider $\mathcal{K} = \mathbb{F}_q((t))$ and $\mathcal{O} = \mathbb{F}[[t]]$ and for a particular example of a cuspidal irreducible representation of $PGL(2,\mathcal{K})$ we review the construction of the isomorphism provided by Theorem \ref{gk_theorem} and Corollary \ref{gk_corollary}. We then rephrase the construction in more geometric terms which allows for a categorical generalization in the next section.

First, let us fix some notation: for an element $x \in \mathcal{O}$ denote by $x_i$ the coefficient of $t^i$ in the expansion $x = \sum\limits_{i=0}^\infty x_it^i$ and by $val(x)$ the minimal integer $i\in \mathbb{Z}$ such that $x_i \neq 0$ (this notation will be used throughout the following two sections). 

\subsection{A construction of a cuspidal irreducible representation of $PGL(2,\mathcal{K})$}

The representation of $G=PGL(2,\mathcal{K})$ which we will be considering will be compactly induced from a one-dimensional representation of a certain subgroup $A$ of $G$. The construction of $A$ is as follows.
\begin{definition}
    Let $I^0 \subset I$ be the unipotent radical of the Iwahori subgroup $I \subset PGL(2,\mathcal{K})$, that is the preimage of the unipotent radical of the standard Borel subgroup under the ``evaluation at zero" map $ev_0:G(\mathcal{O}) \to G$. In explicit terms, $I^0$ is given by 
    \begin{equation}
        I^0 = \left\{ \begin{pmatrix}
            1+ta & b \\
            tc  & 1+td
        \end{pmatrix} \text{ with arbitrary } a,b,c,d \in \mathcal{O} \right\}.
    \end{equation}
    We define the group $A= I^0 \rtimes\mathbb{Z}/2\mathbb{Z} \hookrightarrow PGL(2,\mathcal{K})$ to be the subgroup of $PGL(2,\mathcal{K})$ generated by $I^0$ and $\sigma = \begin{pmatrix}
        0 & 1 
        \\
        t & 0
    \end{pmatrix}$ (note that $\sigma^2=e$ in $PGL_2(\mathcal{K})$).
\end{definition}

\begin{remark}
    Note that
    
    \begin{equation}
    \label{conjunction_by_sigma}
        \sigma \begin{pmatrix}
            1+ta & b \\
            tc  & 1+td
        \end{pmatrix} \sigma^{-1}
    =
    \begin{pmatrix}
            1+td & c \\
            tb  & 1+ta
        \end{pmatrix}
    \end{equation}
    so $\sigma$ normalizes $I^0$.
\end{remark}

\begin{definition}
    Let $\psi:\mathbb{F}_q^+ \to \mathbb{C}^\times$ be a nontrivial character of the additive group of $\mathbb{F}_q$. Define $\chi:I^0 \to \mathbb{C}$ to be a character given by
    \begin{equation}
        \chi \left( \begin{pmatrix}
            1+ta & b \\
            tc  & 1+td
        \end{pmatrix} \right) = \psi(b_0+c_0)
    \end{equation}
    Note that due to \eqref{conjunction_by_sigma} $\chi:I^0 \to \mathbb{C}$ is normalized by $\sigma$ so it also defines a character (with a slight abuse of notation) $\chi:A=I^0 \rtimes\mathbb{Z}/2\mathbb{Z} \to \mathbb{C}$.
\end{definition}

We are now ready to define our representation by compactly inducing it from a one-dimensional representation $\mathbb{C}_\chi$ of $A$ given by the character $\chi:A \to \mathbb{C}$. This can also be said in more concrete terms as follows.

\begin{definition}
        The representation $V = PGL(2,\mathcal{K}) / (A,\chi)$ is a subrepresentation of $C^{\infty}(PGL(2,\mathcal{K})$ formed by functions $f:G(\mathcal{K}) \to \mathbb{C}$ such that for any $a \in A$, any $g \in PGL(2,\mathcal{K})$ and any $f \in V$ we have $f(ga)=\chi(a)f(g)$.
\end{definition}
    
\begin{claim}
\label{cuspidal_irrep}
    The representation $V$ defined above is an irreducible cuspidal representation of $PGL(2,\mathcal{K})$
\end{claim}
This is a well-known fact, but since we don't know a good reference for it and for the sake of completeness we include the proof.

First of all, we will need the following statement, which justifies our decision to pass from $I^0$ to $A$.

\begin{claim}
    We have a decomposition 
    \begin{equation}
    \label{decomposition}
        PGL(2,\mathcal{K})=B(\mathcal{K})\cdot A
    \end{equation}
\end{claim}

\begin{proof}
Let $M = \begin{pmatrix}
    m_{11}&& m_{12}
    \\
    m_{21} && m_{22}
\end{pmatrix}$ with $m_{ij}\in \mathcal{K}$ be a matrix representing some given element of $PGL(2,\mathcal{K})$.

Case 1: Suppose either $m_{21}=0$ or $m_{22}=0$. Then $M$ lies in $B(\mathcal{K})$ or $B(\mathcal{K})\cdot \sigma$ respectively. From now on we assume $m_{21}\neq 0$ and $m_{22}\neq 0$.

Case 2: Suppose that $val(m_{21})>val(m_{22})$  - note that this is a well defined condition for an element of $PGL(2,\mathcal{K})$. In this case we can assume that $m_{22} \in 1+t\mathcal{O}$ and $m_{21}\in t\mathcal{O}$ (recall that we can multiply all of the $m_{ij}$ by the same element of $\mathcal{K}$ without changing the class in $PGL_2(\mathcal{K})$). Consider the decomposition
\begin{equation}
    \begin{pmatrix}
        m_{11}&& m_{12}
    \\
    m_{21} && m_{22}
    \end{pmatrix}
    =
    \begin{pmatrix}
        m_{11}-m_{12} && \frac{m_{12}}{m_{21}}
        \\
        0 && 1
    \end{pmatrix}
    \cdot
    \begin{pmatrix}
        1 && 0
        \\
        m_{21} && m_{22}
    \end{pmatrix}.
\end{equation}
One sees that since $m_{22} \in 1+t\mathcal{O}$ and $m_{21}\in t\mathcal{O}$ the right hand side lies in $B(\mathcal{K})\cdot I^0 \subset B(\mathcal{K})\cdot A$.

Case 3: Suppose that $val(m_{21}) \leq val(m_{22})$. In this case we can assume that $m_{21} = t$ and $m_{21}\in t\mathcal{O}$. With that in mind consider the decomposition
\begin{multline}
    \begin{pmatrix}
        m_{11}&& m_{12}
    \\
    t && m_{22}
    \end{pmatrix}
    =
    \begin{pmatrix}
        m_{12}-\frac{m_{11}m_{22}}{t} && \frac{m_{11}}{t}
        \\
        0 && 1
    \end{pmatrix}
    \cdot
    \begin{pmatrix}
        0 && 1
        \\
        t && 0
    \end{pmatrix}
    \cdot
    \begin{pmatrix}
        1 && m_{22}
        \\
        0 && 1
    \end{pmatrix}
    =
    \\
    =
    \begin{pmatrix}
        m_{12}-\frac{m_{11}m_{22}}{t} && \frac{m_{11}}{t}
        \\
        0 && 1
    \end{pmatrix}
    \cdot
    \begin{pmatrix}
        0 && 1
        \\
        t && m_{22}
    \end{pmatrix}.
\end{multline}

One sees that the right hand side lies in $B(\mathcal{K})\cdot (\sigma \cdot I^0) \subset B(\mathcal{K})\cdot A$.
\end{proof}

\begin{remark}
    In the following proof and below we will often say that an orbit can or cannot support functions, functionals or sheaves with given equivariance conditions. By that we mean that the restriction of such an object to this orbit must be zero (we do not mean that this one orbit was expected to be precisely the support of this object).
\end{remark}

Now we can prove the Claim \ref{cuspidal_irrep}:
\begin{proof}

By definition every element of $V$ is locally constant on $PGL_2(\mathcal{K})$ so $V$ is admissible. Therefore to prove cuspidality of $V$ we only to show that it is quasi-cuspidal, i.e. that any $U(\mathcal{K})$-invariant functional on $V$ is zero. 
We will show that no orbit of $U(\mathcal{K})\times A$ on $PGL_2(\mathcal{K})$ (where $U(\mathcal{K})$ acts on $PGL_2(\mathcal{K})$ by multiplication on the left and $A$ by multiplication on the right) can support functionals that are invariant with respect to $U(\mathcal{K})$ and change under the character $\chi$ with respect to the action of $A$ (that is, a restriction of any such functional to such orbit must be zero). Due to the Decomposition \eqref{decomposition} it is enough to consider orbits of elements of $B(\mathcal{K})$, so take some $g = \begin{pmatrix}
    a && b
    \\
    0 &&1
    \end{pmatrix}\in G$ . We want to describe the stabilizer of $g$ under the action of $U(\mathcal{K}) \times A$, i.e. the set of solutions of the matrix equation $\overline{u}gi=g$ with $\overline{u} \in U(\mathcal{K})$ and $i \in A$. This is equivalent to $g^{-1}\overline{u}g = i$, so we compute
    \begin{multline}
    i=g^{-1}\overline{u}g = \begin{pmatrix}
            \frac{1}{a} && -\frac{b}{a}
            \\
            0 && 1
        \end{pmatrix}
        \begin{pmatrix}
            1 && u
            \\
            0 && 1
        \end{pmatrix} 
        \begin{pmatrix}
            a && b
            \\
            0 && 1
        \end{pmatrix}
        =
        \begin{pmatrix}
            1 && \frac{u}{a}
            \\
            0 && 1
        \end{pmatrix}
    \end{multline}
    Note that if we take $u=a$ then the right hand side lies in $I^0$ and $\chi(i) \neq 0$, so this orbit cannot support functionals which satisfy the necessary  conditions.

    Now let us show that $V$ is irreducible. Recall that the category of cuspidal representations of $G$ is semi-simple, therefore to establish irreducibility of $V$ it is enough to show that $\text{dim} Hom(V,V) \le 1$. Denote (following \cite{gk}) by $C^{-\infty}(A,\chi \backslash G / A, \chi)$ the space of distributions which are equivariant against the character $\chi$  with respect to left and right action of $A$. Since by definition $V =\text{ind}_{A}^{G}(\mathbb{C},\chi)$ we have an embedding 
    \begin{equation}
        Hom(V,V) \hookrightarrow C^{-\infty}(A,\chi \backslash G / A, \chi)
    \end{equation} given by $F \mapsto ev_e \circ F$ where $ev_e$ is the evaluation at the identity element of $G$. Therefore it is enough to show that $\text{dim} \left( C^{-\infty}(A,\chi \backslash G / A, \chi) \right) \le 1$. We do this by showing that there at most one $A \times A$ orbit on $G$ which can support a distribution equivariant against a character $(\chi,\chi)$. 

    Since $\chi(\sigma)=0$ it is enough to show that there is only one $I^0 \times I^0$ orbit on $G$ which can support a distribution equivariant against a character $(\chi,\chi)$. Such orbits are parametrized by a product of the affine Weyl group and the Cartan torus inside $G=PGL_2(\mathcal{K})$, that is by matrices of the form
    \begin{equation}
        \begin{pmatrix}
            at^n & 0
            \\
            0 & 1
        \end{pmatrix} \text{   and   }
        \begin{pmatrix}
            0 && at^n
            \\
            1 && 0
        \end{pmatrix}.
    \end{equation} 
    
    We want to show that only one such matrix has a stabilier in $I^0 \times I^0$ on which the character $(\chi, \chi)$ is trivial. To parametrize the corresponding stabilizer note that describing solutions of $h_1gh_2=g$ with $h_1,h_2 \in I^0$ and $g$ of the form given above is equivalent to describing solutions of $gh_1g^{-1} \in I^0$. 

    Now if $g$ is of the form $\begin{pmatrix}
            at^n & 0
            \\
            0 & 1
        \end{pmatrix}$ we compute

    \begin{equation}
        \begin{pmatrix}
            at^n && 0
            \\
            0 && 1
        \end{pmatrix}
        \begin{pmatrix}
            1+t\alpha && \beta
            \\
            t\gamma && 1+t\varepsilon
        \end{pmatrix}
        \begin{pmatrix}
            a^{-1}t^{-n} && 0
            \\
            0 && 1
        \end{pmatrix}
        = \begin{pmatrix}
            1 + t\alpha && at^n\beta
            \\
            a^{-1}t^{-n+1}\gamma && 1+t\varepsilon
        \end{pmatrix}
    \end{equation}
    which lies in $I^0$ if and only if $val(\gamma) \ge n$. We see that if applying an element of a stabilizer of $g$ must multiply a $(\psi,\psi)$-equivariant distribution by 
    $\psi(\beta^{(0)} + \gamma^{(0)} + a\beta^{(0)} +a^{-1}\gamma^{(0)})$ if $n=0$ and $\psi(\beta^{(0)}+a^{-1}\gamma^{(n)})$ if $n>0$.
    So the only $g$ for which $(\chi,\chi)$ is identically trivial on the corresponding stabilizer is $g=\begin{pmatrix}
        1 && 0
        \\
        0 && 1
    \end{pmatrix}$. 
    
    On the other hand if $g$ has the form $\begin{pmatrix}
        0 && at^n
        \\
        1 && 0
    \end{pmatrix}$
    then

    \begin{equation}
        \begin{pmatrix}
            0 && at^n
            \\
            1 && 0
        \end{pmatrix}
        \begin{pmatrix}
            1+t\alpha && \beta
            \\
            t\gamma && 1+t\varepsilon
        \end{pmatrix}
        \begin{pmatrix}
            0 && at^n
            \\
            1 && 0
        \end{pmatrix}
        = \begin{pmatrix}
            a + at^{n+1}\varepsilon && a^2t^{n+1}\alpha
            \\
            t^{-n}\beta && a+at\alpha
        \end{pmatrix}
    \end{equation}
    which lies in $I^0$ if and only if $val(\beta) \ge n+1$. We see that applying an element of a stabilizer of $g$ must multiply a $(\chi,\chi)$-equivariant distribution by $\psi(\beta^{(0)} + \gamma^{(0)}+ \beta^{(1)})$
    so the character $(\chi,\chi)$ is not trivial on the stabilizer of $g$.

    Therefore, the only $A\times A$-orbit that can support a $(\chi,\chi)$-equivariant distribution is the orbit if the identity element, so 
    \begin{equation}
        \text{dim} Hom(V,V) \le dim C^{-\infty}(A,\chi \backslash G / A, \chi) \le 1
    \end{equation} as required.
\end{proof}

\subsection{Explicit construction of the isomorphism}

We would like now to discuss a geometric construction of the isomorphism $\phi$ of Theorem \ref{gk_theorem} in the case of the representation $V$ defined in the previous subsection. Recall that $\phi$ should send a vector $v \in V$ to a matrix coefficient defined by $v$ and some Whittaker functional $\ell:V \to \mathbb{C}$. Since our representation is realized in a space of functions on $PGL_2(\mathcal{K})$ we can build the nessesary Whittaker functional on it by the means of "Whittaker averaging": namely, for a function $f \in PGL_2(\mathcal{K})/(A,\chi)$ we define $\ell(f)$ to be
\begin{equation}
    \ell(f) = \int\limits_{U(\mathcal{O})}f(u)\chi^{-1}(u)du
\end{equation}

Then the isomorphism $\phi:V \to C^\infty(\mathcal{K})$ is given by
\begin{equation}
\label{phi_definition}
    \phi(f)(x) = \ell(x.f) = \int\limits_{u \in U(\mathcal{O})}f(xu)\chi^{-1}(u)du.
\end{equation}
\begin{remark}
    Note that $\phi$ is a morphism of representations of $\mathcal{K}^\times$.
\end{remark}
In the remaining part of the section we give a geometric proof that this $\phi$ is an isomorphism (recall that the proof of this fact for a general case is given in \cite{bz2}). Here by ``geometric" we mean a construction based on the classification of relevant orbits (i.e. those that can support objects with necessary equivariant conditions) and realization of $\phi$ as a map given by a kernel living on the product of $\mathcal{K}^\times$ and the variety that parametrizes relevant orbits.

Both $V$ and the regular representation of $\mathcal{K}^\times$ are direct limits of spaces of functions which are invariant under the congruence subgroups of $G$ and $\mathcal{K}^\times$ respectively. Therefore aside from the action of $A$ on $G$ by multiplication on the right we have to consider the action of the congruence subgroups $\{(1+t^k\mathcal{O})\}$ of the Cartan torus by multiplication on the left, and this is what we mean by the action of $(1+t^k\mathcal{O})$ on $G$ everywhere below. On the other side of the isomorphism, $(1+t^k\mathcal{O})$ also acts naturally on $\mathcal{K}^\times$. We will be building isomorphism $\phi$ by considering its restrictions to the subspaces of $V$ and $C^\infty(\mathcal{K}^\times)$ consisting of functions invariant under $(1+t^k\mathcal{O})$ for the same $k$ in a coherent (with respect to different $k$) way.

Therefore, let us first classify relevant orbits of $(1+t^k O) \times A$ on $G$. Once again, by ``relevant" here we mean orbits that support functions with necessary equivariance conditions: in this case which are invariant under the action of $(1+t^k \mathcal{O})$ and equivariant against $\chi$ with respect to the right action of $A$). 

The first step in the classification of the relevant orbits is as follows:

\begin{claim}
\label{first_condition_on_relevance}
    Orbit of a matrix $\begin{pmatrix}
        a & b
        \\
        0 & 1
    \end{pmatrix}$ can support functions that are twisted-equivariant with respect to $(1+t^k\mathcal{O})\times A$ with character $(0,\chi)$ if and only if $val(\frac{b}{a}) \in \{-k+1 ,...,-1\}$
\end{claim}
\begin{proof}
    Let 
    \begin{equation}
        g = \begin{pmatrix}
        at^n && b
        \\ 0 && 1
    \end{pmatrix} \text{   and   }h = \begin{pmatrix}
        1 +t^kx && 0 
        \\
        0 && 1
    \end{pmatrix}.
    \end{equation}
    Then a direct calculation shows that
    \begin{equation}
        g^{-1}hg = \begin{pmatrix}
            1+t^kx && t^kx\frac{b}{a}
            \\
            0 && 1
        \end{pmatrix}
    \end{equation}
    therefore to have $\chi(g^{-1}hg)=0$ for all $h \in \mathcal{O}$ we must have $val(\frac{b}{a}) = \{ -k+1,...,-1\}$.
\end{proof}
Although this provides a list of all relevant orbits, it is not the most effective description, as each relevant orbit can have several representatives of the form given above. A better set of representatives is provided by the following lemma.
\begin{lemma}
\label{lemma_representatives}
    Denote by $R_{n,k}$ the set of matrices of the form  
    \begin{equation}
    \label{R_rep}
    R_{a,b}=
        \begin{pmatrix}
        a_nt^n & b_{n-k+1}t^{n-k+1}+ ... + b_{n-1}t^{n-1}
        \\
        0 & 1
    \end{pmatrix}
    \end{equation}
    where $a=a_n\in\mathbb{F}_q^\times$ and $b=b_{n-k+1}t^{n-k+1}+ ... + b_{n-1}t^{n-1}$ with $b_i \in \mathbb{F}_q$. Then $R_{n,k}$ contains a unique representative of every $(1+t^k\mathcal{O})\times A$-orbit which supports functions equivariant with character $(0,\chi)$.
\end{lemma}
\begin{proof}
We know from Claim \ref{first_condition_on_relevance} that in every relevant right $A$-coset there is a Borel matrix $\begin{pmatrix}
        a & b
        \\
        0 & 1
    \end{pmatrix}$ with $val(b)<val(a)$. Now let us provide a general form of a right $I^0$-coset of a Borel matrix:

\begin{equation}
     \begin{pmatrix}
        a & b
        \\
        0 & 1
    \end{pmatrix}
     \begin{pmatrix}
        1+tm_{11} & m_{12}
        \\
        tm_{21} & 1+tm_{22}
    \end{pmatrix}
    = 
     \begin{pmatrix}
        a+tam_{11}+tbm_{21} & am_{12}+b+tbm_{22}
        \\
        tm_{21} & 1+tm_{22}
    \end{pmatrix}.
\end{equation}

Therefore, if we take $m_{21}=0$ and $m_{22}=0$ as well as
\begin{equation}
    m_{11}=\frac{-a+a_{val(a)}t^{val(a)}}{ta}
\end{equation} 
(which lies in $\mathcal{O}$ because $val(a-a_{val(a)}t^{val(a)}) \ge val(a)+1 = val(ta)$) and
\begin{equation}
    m_{12}= -\frac{b_{val(a)}t^{val(a)}+...}{a} \in \mathcal{O}
\end{equation}
then we can turn $\begin{pmatrix}
        a & b
        \\
        0 & 1
    \end{pmatrix}$ with $val(b)<val(a)$ into a matrix of the form
    \begin{equation}
    \label{R_rep}
        \begin{pmatrix}
        a_nt^n & b_{n-k+1}t^{n-k+1}+ ... + b_{n-1}t^{n-1}
        \\
        0 & 1
    \end{pmatrix}
    \end{equation}
    and since such $m_{i,j}$ above are uniquely defined the class $R$ of matrices of the form \eqref{R_rep} contains a unique representative from each relevant orbit.
        
\end{proof}

\begin{remark}
    One can see from the proof that for a Borel matrix $\begin{pmatrix}
        a & b 
        \\
        0 & 1
    \end{pmatrix}$ its representative (modulo $A$) in $R$ is given by 
    \begin{equation}
        \begin{pmatrix}
        a_n t^n & b_{n-k+1}t^{n-k+1}+...+b_{n-1}t^{n-1} 
        \\
        0 & 1
    \end{pmatrix}
    \end{equation}
    where $n=val(a)$ and $k=val(a)-val(b)+1$ (so that $val(b)=n-k+1$).
\end{remark}
\begin{remark}
    Note that we have $R_{n,k}\subset R_{n,k+1}$ and $R_{n_1,k_1}\cap R_{n_2,k_2} = \varnothing$ unless $n_1=n_2$.
\end{remark}

\begin{corollary}
     The restriction maps $res_B:V \to B(\mathcal{K})/(U(\mathcal{O},\chi)) $ and $res_R:V \to C^\infty(R)$ are isomorphisms due to Lemma \ref{lemma_representatives}). 
\end{corollary}

\begin{definition}
    \label{def_phi_B}
    We denote by $\phi_B$ and $\phi_R$ the restrictions of $\phi$ to $B(\mathcal{K})$ and $R$: $\phi_B=\phi \circ (res_B)^{-1}$ and $\phi_R=\phi \circ (res_R)^{-1}$. Note that $\phi_B$ is given by the same formula as $\phi$, except that $f$ is now formally an element of $B(\mathcal{K})/(U(\mathcal{O},\chi))$:
    \begin{equation}
        \phi_B(f)(x) = l(x.f) = \int\limits_{u \in U(\mathcal{O})}f(xu)\chi^{-1}(u)du
    \end{equation}.
\end{definition}

We are ready to proceed with the geometric construction of the isomorphism $\phi: V \xrightarrow{\sim} C^\infty(\mathcal{K}^\times)$. The idea is that we can first restrict functions from $V=PGL_2(\mathcal{K})$ to the variety $R$ and then realize $\phi|_R$ as a transform given by a certain kernel.

We therefore want to find a kernel function $K:R \times \mathcal{K}^\times \to \mathbb{C}$ such that
\begin{equation}
    \phi_R (f) = pr_{2*}(K\cdot pr_1^*(f))
\end{equation}
where $pr_1: R \times \mathcal{K}^\times \to R$ and $pr_2: R \times \mathcal{K}^\times \to \mathcal{K}^\times$ are the natural projections. To descibe such kernal it is enough to find the image of a delta function of an arbitrary point of $R$. 

For this take some $M\in R$, that is $M=\begin{pmatrix}
    a & b
    \\
    0 & 1
\end{pmatrix} \in R$ with $a = a_nt^n \in \mathbb{F}_q^\times t^n$ and $b=b_{n-k+1}t^{n-k+1}+...+b_{n-1}t^{n-1} \in \mathcal{K}^\times$. Let $\delta_M^{(R)}$ be the delta function of $M$ in $R$, that is
\begin{equation}
    \delta_M^{(R)}(N) = \begin{cases}
    1,& \text{if } N = M\\
    0,              & \text{otherwise}
\end{cases}.
\end{equation}

\begin{claim}
\label{kernel_computation}
    The value of $\phi_R(\delta_M^{(R)})$ at a point $x\in \mathcal{K}$ is given by:
    \begin{equation}
         \phi_R(\delta_M^{(R)})(x)= 
\begin{cases}
    -q \left( \frac{b}{x}\right)_0,& \text{if } \frac{x}{a} \in 1+t\mathcal{O}\\
    0,              & \text{otherwise}
\end{cases}
    \end{equation}
    where $q=vol(\mathcal{O})$.
\end{claim}

\begin{proof}

Under the isomorphism
\begin{equation}
    PGL_2(\mathcal{K})/(A,\chi) \simeq B(\mathcal{K})/(U(\mathcal{O},\chi)) \simeq C^\infty(R)
\end{equation}
the function $\delta_M^{(R)}$ corresponds to a function $\delta_M^{(B)} \in B(\mathcal{K})/(U(\mathcal{O},\chi))$ given by 
\begin{equation}
    \delta_M^{(B)}(N) = \begin{cases}
    \chi(u),& \text{if } N = M\cdot u \text{ with } u\in U(\mathcal{O})\\
    0,              & \text{otherwise}
\end{cases}.
\end{equation}
By definition

\begin{equation}
\label{p_expression}
    \phi_R(\delta_M^{(R)})(x) =\phi_B(\delta_M^{(B)})(x) = \int\limits_{u \in U(\mathcal{O})} \delta_M^{(B)}(xu)\chi^{-1}(u)du = \int\limits_{S} \chi(s(u))\chi^{-1}(u)du
\end{equation}

where
\begin{equation}
    S = \big\{ u \in U(\mathcal{O}) \Big| xu=M\cdot s(u) \text{ with } s(u) \in A \big\}.
\end{equation}

Condition $xu=M\cdot s(u) \text{ with } s(u) \in A$ is equivalent to
\begin{equation}
\label{m_id_1}
    \begin{pmatrix}
    a & b
    \\
    0 & 1
\end{pmatrix}
\begin{pmatrix}
    \alpha & \beta
    \\
    0 & 1
\end{pmatrix}
=
\begin{pmatrix}
    x & xu
    \\
    0 & 1
\end{pmatrix}.
\end{equation}

Now in \eqref{p_expression} we are given  $x \in \mathcal{K}^\times$, $a = a_nt^n \in \mathbb{F}_q^\times t^n$ and $b=b_{n-k+1}t^{n-k+1}+...+b_{n-1}t^{n-1} \in \mathcal{K}^\times$ and we need to describe all $u \in\mathcal{O}$ for which there exist $\alpha \in 1+t\mathcal{O}$ and $\beta \in \mathcal{O}$ such that \eqref{m_id_1} is satisfied. The conditions on all these variables coming from \eqref{m_id_1} are $a \cdot \alpha=x$ and $a\beta +b=xu$. From this we see that:
\begin{itemize}
    \item $ \phi_R(\delta_M^{(R)})(x)=0$ unless $\frac{x}{a} \in 1+t\mathcal{O}$.
    \item If $\frac{x}{a} \in 1+t\mathcal{O}$, then we are forced to set $\alpha = \frac{x}{a}$. Then the remaining condition is
    $a\beta +b=xu$ which is the same as $u = \frac{a}{x}\beta+ \frac{b}{x}$. That is, $u$ is parametrized by $\beta \in \mathcal{O}$ and in such a way that
    \begin{equation}
        \chi\left( \begin{pmatrix}
    1 & u(\beta)
    \\
    0 & 1
\end{pmatrix} \right) = \left( u(\beta) \right)_0=\beta_0+\left( \frac{b}{x}\right)_0.
    \end{equation}
\end{itemize}

Therefore, the integral in \eqref{p_expression} can be calculated as an integral over $\beta \in \mathcal{O}$:

\begin{multline}
    \phi_R(\delta_M^{(R)})(x) = \int\limits_{u \in U(\mathcal{K})} \delta_M^{(R)}(xu)\chi^{-1}(u)du = \int\limits_{\beta\in \mathcal{O}} \chi(M^{-1}\cdot x \cdot u(\beta))\chi^{-1}(u(\beta))d\beta =
    \\
    =
    \int\limits_{\beta \in \mathcal{O}} \chi \left( 
    \begin{pmatrix}
        \frac{1}{a} & -\frac{b}{a} 
        \\
        0 & 1
    \end{pmatrix}
    \begin{pmatrix}
        x & 0 
        \\
        0 & 1
    \end{pmatrix}
    \begin{pmatrix}
        1 & u(\beta)
        \\
        0 & 1
    \end{pmatrix}
    \right)
    \chi^{-1} \left(
    \begin{pmatrix}
        1 & u(\beta)
        \\
        0 & 1
    \end{pmatrix}
    \right)
    d\beta 
\end{multline}
and since
\begin{multline}
    \begin{pmatrix}
        \frac{1}{a} & -\frac{b}{a} 
        \\
        0 & 1
    \end{pmatrix}
    \begin{pmatrix}
        x & 0 
        \\
        0 & 1
    \end{pmatrix}
    \begin{pmatrix}
        1 & u(\beta)
        \\
        0 & 1
    \end{pmatrix} 
    =
    \begin{pmatrix}
        \frac{1}{a} & -\frac{b}{a} 
        \\
        0 & 1
    \end{pmatrix}
    \begin{pmatrix}
        x & xu(\beta)
        \\
        0 & 1
    \end{pmatrix} =
    \\
    = 
    \begin{pmatrix}
        \frac{1}{a} & -\frac{b}{a} 
        \\
        0 & 1
    \end{pmatrix}
    \begin{pmatrix}
        x & a\beta+b
        \\
        0 & 1
    \end{pmatrix}
    =
    \begin{pmatrix}
        \frac{x}{a} & \beta + \frac{b}{a} - \frac{b}{a}
        \\
        0 & 1
    \end{pmatrix}
    =\begin{pmatrix}
        \frac{x}{a} & \beta
        \\
        0 & 1
    \end{pmatrix},
\end{multline}
we have (if $\frac{x}{a} \in 1+t\mathcal{O}$) the following expression for $\phi_R(\delta_M^{(R)})(x)$:
\begin{equation}
     \phi_R(\delta_M^{(R)})(x) = \int\limits_{\beta \in \mathcal{O}} ( \beta_0 - u(\beta))_0 ) d\beta=  q \left( \beta_0-\beta_0- \left( \frac{b}{x}\right)_0 \right)= -q \left( \frac{b}{x}\right)_0
\end{equation}
where $q=vol(\mathcal{O})$.
\end{proof}
Knowing the image of a delta function under $\phi_R$ allows us to realize $\phi$ as an a tranform given by a kernel:
\begin{corollary}
\label{kernel}
    The isomorphism $\phi:V=PGL_2(\mathcal{K})/(A\chi)\to C^\infty(\mathcal{K}^\times)$ of Theorem \ref{gk_theorem} can be realized as 
    \begin{equation}
        \phi=\phi_R \circ res_R
    \end{equation} where 
    \begin{equation}
        res_R:PGL_2(\mathcal{K})/(A\chi)\to C^\infty(R)
    \end{equation} is the restriction and 
        \begin{equation}
            \phi_R = pr_{2*}(K\cdot pr_1^*(\bullet))
        \end{equation}
with $pr_1: R \times \mathcal{K}^\times \to R$ and $pr_2: R \times \mathcal{K}^\times \to \mathcal{K}^\times$ the natural projections and the kernel function $K:R \times \mathcal{K}^\times \to \mathbb{C}$ is given by
\begin{equation}
\label{formula_kernel}
    K: \begin{pmatrix}
        a && b
        \\
        0 && 1
    \end{pmatrix}\times x \mapsto \begin{cases}
    -q \left( \frac{b}{x}\right)_0,& \text{if } \frac{x}{a} \in 1+t\mathcal{O}\\
    0,              & \text{otherwise}
\end{cases}.
\end{equation}
\end{corollary}

\section{Categorification of the isomorphism}
We will now discuss a categorified version of the above isomorphism. Namely, we will upgrade it to an equivalence of categories of $D$-modules on our groups that satisfy analogous equivariant conditions. While we do not currently now how to adapt the reasonings of \cite{gk} and \cite{bz2} to categorify Theorem \ref{gk_theorem} in full generality, in our particular case we can use the "geometric" construction of $\phi$ provided in the previous section.

For the general discussion of the theory of $D$-modules on loop groups we refer to \cite{r1}.

\begin{definition}
   Our category $\mathcal{C} = D(PGL_2(\mathcal{K}))_{/(A,\chi)}$ is the category of $D$-modules on $PGL_2(\mathcal{K})$ that are $A$-equivariant against $\chi$.
\end{definition}

Recall the following definition discussed in the introduction.
\begin{definition}
    We say that a category $\mathcal{A}$ equipped with a strong action of a reductive group $G$ is a cuspidal irreducible categorical representation of $G$ if:
    \begin{enumerate}
        \item it is cuspidal, meaning that the invariants with respect to the action of $U(\mathcal{K})$ are trivial and
        \item the Whittaker invariants functor provides an equivalence between $\mathcal{A}$ and the category $Vect$ of vector spaces.
    \end{enumerate}
\end{definition}
\begin{claim}
    The category $\mathcal{C}$ defined above is a cuspidal irreducible categorical representation.
\end{claim}

\begin{proof}
    First, let us show that it is cuspidal. Consider the action of $U(\mathcal{K}) \times A$ on $PGL_2$ given by the action of $U(\mathcal{K})$ on the left and $A$ on the right. We will show that no orbit can support sheaves equivariant with resect to this action against a character $(0,\chi)$ of $U(\mathcal{K}) \times A$. The proof of this is identical to the proof of cuspidality of $V$ in Claim \ref{cuspidal_irrep}. 

    \iffalse
    
    Due to the Decomposition \eqref{decomposition} it is enough to consider orbits of elements of the Borel subgroup, so take some $g = \begin{pmatrix}
    a && b
    \\
    0 &&1
    \end{pmatrix}\in G$ . We want to describe the stabilizer of $g$ under the action of $U(\mathcal{K}) \times A$, i.e. the set of solutions of the matrix equation $\overline{u}gi=g$ with $\overline{u} \in U(\mathcal{K})$ and $i \in A$. This is equivalent to $g^{-1}\overline{u}g = i$, so we compute
    \begin{multline}
    i=g^{-1}\overline{u}g = \begin{pmatrix}
            \frac{1}{a} && -\frac{b}{a}
            \\
            0 && 1
        \end{pmatrix}
        \begin{pmatrix}
            1 && u
            \\
            0 && 1
        \end{pmatrix} 
        \begin{pmatrix}
            a && b
            \\
            0 && 1
        \end{pmatrix}
        =
        \begin{pmatrix}
            1 && \frac{u}{a}
            \\
            0 && 1
        \end{pmatrix}
    \end{multline}
    Note that if we take $u=a$ then the right hand side lies in $I^0$ and $(0,\chi)(i) = \chi(i) \neq 0$, so this orbit cannot support $D$-modules which satisfy necessary equivariance conditions.

    \fi

    Now let us see that this representation is irreducible. As discussed above, it means that the functor of Whittaker (co)invariants provides an equivalence with the category $Vect$ of vector spaces over $\mathbb{C}$. It is enough to show that there is a unique orbit that supports sheaves $U(\mathcal{K}) \times A$-equivariant against the character $\chi-u_0$. As before, it is enough to consider orbits of elements of the Borel subgroup, and from the computation above we see that the character is trivial on the stabilizer of $g$ if and only if $g$ is the identity matrix.
\end{proof}

Our goal is to build an equivalence of categories
\begin{equation}
    P:D(PGL_2(\mathcal{K}))_{/(A,\chi)} \simeq D(\mathcal{K}^\times).
\end{equation}

\begin{remark}
    We hope that this notation will not confuse the reader, as the letter $P$ was used to denote the mirabolic subgroup in the introduction (but wasn't used in the main body of the text untill now).
\end{remark}
The construction of $P$ is based on the discussion of the previous section:
\begin{definition}
    Given a $(A,\chi)$-equivariant sheaf $\mathcal{F}$ on $PGL_2(\mathcal{K})$, we first restrict it to the Borel subgroup and then apply Whittaker-averaging over $U(\mathcal{O})$ with respect to the character $\chi^{-1}$:
\begin{equation}
    P (\mathcal{F}) = Avg^{(U)}_{\chi^-1} \circ res_B (\mathcal{F}).
\end{equation}
One easily sees that this expression mimics the realization of $\phi$ as $\phi=\phi_B\circ res_B$ from Definition \ref{def_phi_B} of the previous section.

\begin{remark}
\label{remark_compact_objects}
A key technical point is that to establish that this functor is an equivalence it is enough to check this on the level of subcategories formed by $D$-modules equivariant with respect to the action of $k$-th congruence subgroups on both sides for each $k$. The reason is that the categories in question are compactly generated and $D$-modules equivariant with respect to the congruence subgroups are precisely the compact objets on both sides. For this reason every statement below is formulated in terms of some fixed (but arbitrary) $k$.
\end{remark}

\end{definition}

\subsection{Decomposition of the category into a direct product}

One of the subtle points in the categorical setting is that while we will still be able to build an equivalence on "pieces" that correspond to the final-dimensional varieties (corresponding to the presentation of our groups as ind-schemes), we need to put some work to justify that this gives us an equivalence of the full categories. Namely, $D(_{(1+t^k \mathcal{O})} \backslash \mathcal{K}^\times)$ is a direct product of categories $D(t^n\mathcal{O}^\times)$ and we have to show that the analogous decomposition holds for $D(_{(1+t^k \mathcal{O})} \backslash PGL_2(\mathcal{K}))_{/(A,\chi)}$: namely, it is a direct product of its subcategories consisting of $D$-modules supported on those orbits whose representatives in $R$ lie in $R_n = \bigcup\limits_k R_{n,k}$. 

For this we need the following theorem.

\begin{theorem}
\label{direct_product_theorem}
    The morphism $m: (1+t^k\mathcal{O}) \times R_{n,k}\times A \to PGL_2$ given by multiplication in $PGL_2$ is a closed embedding.
\end{theorem}
\begin{proof}
    Since this morphism is injective it is enough to check the valuative criterion of closedness. The schemes in question are of finite type over an algebraically closed field $\mathbb{C}$, so it is enough to check the criterion for $O = \mathbb{C}[[\lambda]]$ and $\mathcal{K}=Frac(O)=\mathbb{C}((\lambda))$.

    So, let $z \in (1+t^k\mathcal{O})$, $r \in R(\mathcal{K})$ and $w \in I^0(\mathcal{K})$ be such that for their product we have $z \cdot r\cdot w = m(z,r,w) \in PGL_2(O)$. We will show that then necessarily $r \in R(O)$ and $w \in I^0(O)$.  Indeed, if
    \begin{equation}
        r = \begin{pmatrix}
            a && b
            \\
            0 && 1
        \end{pmatrix}
    \end{equation}
    where $a \in \mathbb{C}((\lambda))$, $b = b^{(n-k+1)}t^{n-k+1} + ... + b^{(n-1)}t^{n-1} \in \mathbb{C}((\lambda))((t))$ and
    \begin{equation}
        w = \begin{pmatrix}
            1+t\alpha && \beta
            \\
            t\gamma && 1+t \varepsilon
        \end{pmatrix}
    \end{equation}
    where $\alpha, \beta,\gamma,\varepsilon \in \mathbb{C}((\lambda))[[t]]$ and $\alpha^{(i)}, \beta^{(i)},\gamma^{(i)},\varepsilon^{(i)} \in \mathbb{C}((\lambda))$ are the  coefficients of $t^i$ in the corresponding element of $\mathbb{C}((\lambda))[[t]]$.
    Then 
    \begin{equation}
        z\cdot r\cdot w =  \begin{pmatrix}
            a(1+t\alpha)z+b\cdot (t\gamma)z && a\cdot \beta\cdot z + b\cdot (1+t \varepsilon)\cdot z
            \\
            t\gamma && 1+t \varepsilon
        \end{pmatrix} = : \begin{pmatrix}
            m_{11} && m_{12}
            \\
            m_{21} && m_{22}
        \end{pmatrix}.
    \end{equation}
    
    By assumption, this product lies in $PGL_2(\mathbb{C}[[\lambda]][[t]])$.
    First of all, this means that $t\gamma$ and $1+t \varepsilon$ are both elements of $\mathbb{C}[[\lambda]][[t]])$.

    Then consider the upper-right coefficient $m_{12}$ of the product. Recall that $val(a)=n$, $val(\beta) \ge 0$ and $val ( 1+t\varepsilon) =0$, while $b = b^{(n-k+1)}t^{n-k+1} + ... + b^{(n-1)}t^{n-1}$ and $z \in (1+t^k\mathcal{O})$. One easily sees that it implies that $b \in \mathbb{C}[[\lambda]][[t]]$. Indeed, suppose there is $j \in \mathbb{Z}$ which is the minimal integer such that $b^{(i)} \notin \mathbb{C}[[\lambda]]$. Then the coefficient of $t^j$ in $m_{12}$ is equal to 
    \begin{equation}
        m_{12}^{(j)} = b^{(j)} + b^{(j-1)}\varepsilon{(0)} + ... + b^{(n-k+1)}\varepsilon^{(j-n+k-1)}
    \end{equation}
    but we already know that $\varepsilon \in \mathbb{C}[[\lambda]][[t]])$, by assumption $m_{12} \in \mathbb{C}[[\lambda]][[t]])$ and all of the $b^{(i)} \in \mathbb{C}[[\lambda]][[t]])$ for $i<j$ due to our choice of $j$. This is a contradiction, therefore $b \in \mathbb{C}[[\lambda]][[t]]$.

    Finally, consider $m_{11} = a(1+t\alpha)z+b\cdot (t\gamma)\cdot z$ which by assumption lies in $\mathbb{C}[[\lambda]][[t]]$. We also already know that $b$ and $\gamma$ are elements of $\mathbb{C}[[\lambda]][[t]]$. Therefore $a(1+t\alpha)z = m_{11}- b\cdot(t\gamma)z \in \mathbb{C}[[\lambda]][[t]]$ so $a \in \mathbb{C}[[\lambda]]$.

    Since $a \in \mathbb{C}[[\lambda]]$ and $b \in \mathbb{C}[[\lambda]][[t]]$, we see that $r$ actually lies in $PGL_2(\mathbb{C}[[\lambda]])$. By assumption, so does $r\cdot w$ therefore so does $w$.
    \end{proof}

\begin{corollary}
\label{decomposition_corollary}
    The category $\mathcal{C} = D(_{(1+t^k \mathcal{O})} \backslash PGL_2)_{/(A,\chi)}$ is a direct product over $n \in \mathbb{Z}$ of subcategories $\mathcal{C}_n$ consisting of objects supported on orbits whose representatives in $R$ belong to $R_n = \bigcup\limits_k R_{n,k} \subset R$.
\end{corollary}

\subsection{Proving the equivalence}

In this subsection we finish the proof of our main theorem, showing that the functor $P$ defined above is an equivalence of categories. 

We are going to utilize the discussion of the ``geometric" construction of the classical isomorphism $\phi$ provided in Section 1. Namely, we show that the restriction of $P$ to a subcategory $\mathcal{C}_n$ (defined in the Corollary \ref{decomposition_corollary}) can be realized as a Fourier-Deligne transform defined by a kernel $K$ which we calculated in the Corollary \ref{kernel}.

 For that purpose consider the following diagram
    \begin{equation}
    \label{the_diagram}
       % https://q.uiver.app/#q=WzAsNixbMCwwLCJcXG1hdGhiYntHfV9tXFx0aW1lc1xcbWF0aGJie0F9XntrLTF9XFx0aW1lc1xcbWF0aGJie0F9XntrLTF9XFx0aW1lc1xcbWF0aGNhbHtPfSJdLFsyLDAsIlJfe24sa31cXHRpbWVzIEleMCJdLFsyLDIsIlBHTF8yIl0sWzAsMiwiXFxtYXRoY2Fse0t9XlxcdGltZXNfe24sa30gXFx0aW1lcyBVKFxcbWF0aGNhbHtLfSkiXSxbMCw0LCJcXG1hdGhjYWx7S31eXFx0aW1lc197bixrfSJdLFs0LDAsIlJfe24sa30iXSxbMCwzLCJ1ID0gXFxmcmFje2F9e3h9XFxiZXRhK1xcZnJhY3tifXt4fSIsMCx7InN0eWxlIjp7InRhaWwiOnsibmFtZSI6Imhvb2siLCJzaWRlIjoidG9wIn19fV0sWzMsMiwibV8yIiwwLHsic3R5bGUiOnsidGFpbCI6eyJuYW1lIjoiaG9vayIsInNpZGUiOiJ0b3AifX19XSxbMSwyLCJtXzEiLDAseyJzdHlsZSI6eyJ0YWlsIjp7Im5hbWUiOiJob29rIiwic2lkZSI6InRvcCJ9fX1dLFswLDEsIlxcYWxwaGE9XFxmcmFje3h9e2F9IiwwLHsic3R5bGUiOnsidGFpbCI6eyJuYW1lIjoiaG9vayIsInNpZGUiOiJ0b3AifX19XSxbMyw0LCJwcl8xIiwyXSxbMSw1LCJwcl9SIl1d
\begin{tikzcd}
	{\mathbb{G}_m\times\mathbb{A}^{k-1}\times\mathbb{A}^{k-1}\times\mathcal{O}} && {R_{n,k}\times A} && {R_{n,k}} \\
	\\
	(1+t^k\mathcal{O})\backslash B(\mathcal{K}) && (1+t^k\mathcal{O})\backslash  {PGL_2(\mathcal{K})} \\
	\\
	{\mathcal{K}_{n,k}}
	\arrow["{\phi_2}", hook, from=1-1, to=3-1]
	\arrow["{m_2}", hook, from=3-1, to=3-3]
	\arrow["{m_1}", hook, from=1-3, to=3-3]
	\arrow["{\phi_1}", hook, from=1-1, to=1-3]
	\arrow["{pr_1}"', from=3-1, to=5-1]
	\arrow["{pr_R}", from=1-3, to=1-5]
\end{tikzcd}
    \end{equation}

Here $\mathcal{K}_{n,k} := (t^n\mathcal{O}^\times)/(1+t^{k}\mathcal{O}) \simeq \mathbb{G}_m \times \mathbb{A}^{k-1}$ and $\mathbb{G}_m\times\mathbb{A}^{k-1}\times\mathbb{A}^{k-1}\times\mathcal{O}$ with coordinates $(a,b,[\frac{x}{a}]_{1+t^k\mathcal{O}},\beta)$ in the upper-left corner is a presentation of the fiber product $B(\mathcal{K}) \times_{PGL_2}(R_{n,k}\cdot A)$ coming from (\ref{m_id_1}). By $[\frac{x}{a}]_{1+t^k\mathcal{O}}$ we mean a representative of a class in $(t^n\mathcal{O}^\times)/(1+t^{k}\mathcal{O})$, the space of such representatives can be naturally identified with $\mathbb{A}^{k-1}$. The morphisms $\phi_1$ and $\phi_2$ are as follows:
\begin{equation}
    \phi_1 : (a,b,[\frac{x}{a}]_{1+t^k\mathcal{O}},\beta) \mapsto 
    \begin{pmatrix}
    a & b
    \\
    0 & 1
    \end{pmatrix}
    \times
    \begin{pmatrix}
    \frac{x}{a} & \beta
    \\
    0 & 1
\end{pmatrix} \in R_{n,k}\times A
\end{equation}
and 
\begin{equation}
    \phi_2 : (a,b,[\frac{x}{a}]_{1+t^k\mathcal{O}},\beta) \mapsto [\frac{x}{a}]_{1+t^k\mathcal{O}} \times \begin{pmatrix}
    1 & \frac{a}{x}\beta+\frac{b}{x}
    \\
    0 & 1
\end{pmatrix} \in \mathcal{K}_{n,k} \times U(\mathcal{K}).
\end{equation}

Note that in the definitions above the same letters appear as coordinates on different subvarieties. For example, $\beta$ is both a ``coordinate" (an element of $\mathcal{O})$ on $\mathbb{G}_m\times\mathbb{A}^{k-1}\times\mathbb{A}^{k-1}\times\mathcal{O}$ and on $R_{n.k}\times A$ (namely, the upper-right element of a matrix from $A$). This creates a bit of ambiguity, but we leave it this way because it fits well with the matrix computations (like in \ref{m_id_1}) on which this diagram is based and is somewhat more concise than it would be with, say, additional indexes. When in the computations below we use expressions like $\beta_0^* \mathcal{L}_\chi$ we understand $\beta_0$ as a morphism to $\mathbb{A}^1$ (which sends $\beta \in \mathcal{O}$ to its constant term) under which we are taking a pullback of a sheaf $\mathcal{L}_\chi$, and it can always be understood from the context on which variety do we want the resulting to live.

We are finally ready to prove our main theorem:
\begin{theorem}
\label{main_theorem}
    The functor $P:D(PGL_2(\mathcal{K}))_{/(A,\chi)} \to D(\mathcal{K}) $ defined above is an equivalence of categories.
\end{theorem}
\begin{proof}
    Denote by $\mathcal{L}_\chi$ the exponential $D$-module on $\mathbb{A}^1$ defined by the character $\chi$. 
    
    By definition,
    \begin{equation}
    P = Avg^{(U)}_{\chi^-1} \circ res_B  = pr_{1 \; *}\left( (u_0)^{*}\mathcal{L}_{-\chi} \otimes m_2^{*} (\bullet) \right)
\end{equation}

Due to the Corollary \ref{decomposition_corollary} and Remark \ref{remark_compact_objects} it is actually enough to show that $P$ provides an equivalence between the subcategories of $(1+t^k\mathcal{O})$-equivariant objects supported on the subvarieties $R_{n}\cdot A \subset PGL_2(\mathcal{K})$ and $\mathcal{K}_n := t^n\mathcal{O}^\times \subset \mathcal{K}$ respectively. There is a natural equvalence between the category of $(1+t^k\mathcal{O})$-equivariant $D$-modules on $\mathcal{K}_n$ and the category $D(\mathcal{K}_{n,k})$ on $\mathcal{K}_{n,k}$. Also, every $(1+t^k\mathcal{O})$-equivariant object the category $D(PGL_2(\mathcal{K}))_{/(A,\chi)}$ supported on the subvariety $R_{n}\cdot A \subset PGL_2(\mathcal{K})$ comes as an image of a $D$-module on $R_{n,k}$ under the functor
\begin{equation}
    D(R_{n,k}) \ni \mathcal{F} \mapsto \Tilde{\mathcal{F}} := m_{1*}(\mathcal{F} \boxtimes ( \beta_0)^{*}\mathcal{L}_\psi) \in D(PGL_2(\mathcal{K}))_{/(A,\chi)}
\end{equation}
so it is enough to show that the functor $D(R_{n,k}) \to D(\mathcal{K}_{n,k})$ which sends $\mathcal{F} \in D(R_{n,k})$ to 
\begin{equation}
    P(\Tilde{\mathcal{F}}) = pr_{1 \; *}\left( (u_0)^{*}\mathcal{L}_{-\psi} \otimes m_2^{*}m_{1*}(\mathcal{F} \boxtimes ( \beta_0)^{*}\mathcal{L}_\psi)  \right)
\end{equation}
is an equivalence.

Tracing these functors on the diagram \ref{the_diagram} one sees that due to "closed base change" we have 
    \begin{multline}
        P(\Tilde{\mathcal{F}}) = pr_{1 \; *}\left( (u_0)^{*}\mathcal{L}_{-\psi} \otimes m_2^{*}m_{1*}(\mathcal{F} \boxtimes ( \beta_0)^{*}\mathcal{L}_\psi)  \right) \simeq
        \\
        \simeq
        pr_{1 \; *}\left( (u_0)^{*}\mathcal{L}_{-\psi} \otimes \phi_{2*}\phi_{1}^*(\mathcal{F} \boxtimes ( \beta_0)^{*}\mathcal{L}_\psi)  \right) 
        \end{multline}
        which in turn is isomorphic to
        \begin{multline}
        \simeq (pr_1 \circ \phi_2)_* 
        \left( (\beta+\frac{b}{x})_0^*\mathcal{L}_{-\psi} \otimes \phi_{1}^*(\mathcal{F} \boxtimes ( \beta_0)^{*}\mathcal{L}_\psi)  \right) \simeq
        \\
        \simeq 
        (pr_1 \circ \phi_2)_*
        \left(
        (\beta+\frac{b}{x})_0^*\mathcal{L}_{-\psi} \otimes ( \beta_0)^{*}\mathcal{L}_\psi \otimes (pr_R \circ \phi_{1})^*\mathcal{F}
        \right)
        \simeq
        \\
         \simeq (pr_1 \circ \phi_2)_*
        \left(
        K^*\mathcal{L}_{-\psi} \otimes (pr_R \circ \phi_{1})^*\mathcal{F}
        \right)
    \end{multline}
Where $K$ is given by the formula (\ref{formula_kernel}).
    
Now note that:
\begin{enumerate}
    \item The map 
    \begin{equation}
        p_1= pr_R \circ \phi_{1}: \mathbb{G}_m\times\mathbb{A}^{k-1}\times\mathbb{A}^{k-1}\times\mathcal{O} \to R_{n,k} \simeq \mathbb{G}_m\times \mathbb{A}^1
    \end{equation}
    sends the tuple $(a,b,[\frac{x}{a}]_{1+t^k\mathcal{O}},\beta)$ to $\begin{pmatrix}
        a && b
        \\
        0 && 1
    \end{pmatrix}$
    \item The map 
    \begin{equation}
        p_2 = pr_1 \circ \phi_2: \mathbb{G}_m\times\mathbb{A}^{k-1}\times\mathbb{A}^{k-1}\times\mathcal{O} \to \mathcal{K}_{n,k} \simeq \mathbb{G}_m\times \mathbb{A}^1
    \end{equation}
    sends the tuple $(a,b,[\frac{x}{a}]_{1+t^k\mathcal{O}},\beta)$ to $[x]_{1+t^k\mathcal{O}}= a\cdot [\frac{x}{a}]_{1+t^k\mathcal{O}}$
    \item $K(a,b,[x])_{1+t^k\mathcal{O}}$ defines a morphism $\mathbb{G}_m\times\mathbb{A}^{k-1}\times\mathbb{A}^{k-1} \to \mathbb{A}^1$ which is nonzero if and only if $\frac{x}{a} \in $, in which case is $K(a,b,[x])= (\frac{b}{x})_0$. 
    \item From the expressions above one sees that the maps $p_1$ and $p_2$ commute with projections to the $\mathbb{G}_m$, so the diagram
    % https://q.uiver.app/#q=WzAsMyxbMCwwLCJcXG1hdGhiYntBfV57bi0xfVxcdGltZXNcXG1hdGhiYntBfV57bi0xfVxcdGltZXNcXG1hdGhjYWx7T30iXSxbMCwyLCJcXG1hdGhiYntBfV57bi0xfSJdLFsyLDAsIlxcbWF0aGJie0F9XntuLTF9Il0sWzAsMiwicF8xIiwyXSxbMCwxLCJwXzIiXV0=
\[\begin{tikzcd}
	{\mathbb{G}_m \times \mathbb{A}^{k-1}\times\mathbb{A}^{k-1}\times\mathcal{O}} && {\mathbb{G}_m \times \mathbb{A}^{k-1}_b} \\
	\\
	{\mathbb{G}_m \times \mathbb{A}^{k-1}_{[\frac{x}{a}]}}
	\arrow["{p_1}"', from=1-1, to=1-3]
	\arrow["{p_2}", from=1-1, to=3-1]
\end{tikzcd}\]

can be viewed as a diagram of (trivial) vector bundles over $\mathbb{G}_m$. The factor of $\mathcal{O}$ can be ignored because from the expression for $p_2$ above we see that the only effect it has is multiplication by the cohomology of the pushforward of the structure sheaf of $\mathcal{O}$, which is a one-dimensional vector space since $\mathcal{O}$ is contractible. 

\item The map $K: \mathbb{G}_m\times\mathbb{A}^{k-1}\times\mathbb{A}^{k-1} \to \mathbb{A}^1$ (note that this is well-defined since $K$ ignores the factor of $\mathcal{O}$) then induces a fiberwise nondegenerate pairing which identifies fibers of the upper-right bundle with duals to the fibers of the bottom bundle. Indeed, a change of coordinate $z=\frac{1}{x}$ takes the form $K(b,z)=(b,z)_0$, which with our assumptions on $b$ and $z$ (recall that $x\in at^n\mathcal{O}^\times/(1+t^k\mathcal{O})$ and $b= b_{n-k+1}t^{n-k+1}+ ... + b_{n-1}t^{n-1}$) is easily seen to be a nondegenerate pairing. 

\item Therefore we see that our functor is just a Fourier-Deligne transform between vector bundles over $\mathbb{G}_m$ with respect to the kernel given by $K$, hence it is an equivalence.
\end{enumerate}

\end{proof}

\end{document}